\documentclass[12pt,a4paper]{amsart}

\usepackage[utf8]{inputenc}
\usepackage{enumerate,bbm,titletoc,a4,mathtools,mathrsfs,amssymb,xspace,amsthm}
\usepackage[usenames,dvipsnames]{xcolor}
\usepackage{tikz-cd}
\usepackage[colorlinks,linkcolor=BrickRed,citecolor=CadetBlue,urlcolor=black,hypertexnames=true]{hyperref}

\DeclareMathOperator{\ran}{ran}
\DeclareMathOperator{\rk}{rk}
\DeclareMathOperator{\GL}{GL}
\DeclareMathOperator{\SL}{SL}

\DeclareMathOperator{\dom}{dom}
\DeclareMathOperator{\grad}{grad}
\DeclareMathOperator{\spa}{span}
\DeclareMathOperator{\tr}{tr}
\DeclareMathOperator{\opm}{M}
\DeclareMathOperator{\re}{re}

\newcommand{\ve}{\varepsilon}

\newcommand{\de}{\delta}
\newcommand{\bA}{\mathbb{A}}
\newcommand{\N}{\mathbb{N}}
\newcommand{\Z}{\mathbb{Z}}
\newcommand{\R}{\mathbb{R}}
\newcommand{\C}{\mathbb{C}}

\newcommand{\cC}{\mathcal{C}}

\newcommand{\cS}{\mathcal{S}}

\newcommand{\cV}{\mathcal{V}}
\newcommand{\cW}{\mathcal{W}}

\newcommand{\cZ}{\mathcal{Z}}
\newcommand{\JJ}{\mathscr{J}}
\newcommand{\cJ}{\JJ}

\newcommand{\rr}{\mathbbm r}
\newcommand{\ti}{{\rm t}}

\newcommand{\gx}{\omega}
\newcommand{\gy}{\upsilon}
\newcommand{\gX}{\Omega}
\newcommand{\gY}{\Upsilon}

\newcommand{\fl}{\mathscr{Z}}
\newcommand*{\mat}[1]{\opm_{#1}(\C)}

\newcommand{\Langle}{\mathop{<}\!}
\newcommand{\Rangle}{\!\mathop{>}}
\newcommand{\px}{\C\!\Langle x,x^*\Rangle}
\newcommand{\ax}{\C\!\Langle x\Rangle}
\newcommand{\ay}{\C\!\Langle y\Rangle}
\newcommand{\tx}{\C\!\Langle \tilde{x}\Rangle}

\newcommand{\axy}{\C\!\Langle x,y\Rangle}
\newcommand{\pxy}{\C\!\Langle x,x^*,y,y^*\Rangle}
\newcommand{\axyy}{\C\!\Langle x,y,y'\Rangle}
\newcommand{\gxy}{\C\!\Langle x,y,y^{-1}\Rangle}

\makeatletter
\def\moverlay{\mathpalette\mov@rlay}
\def\mov@rlay#1#2{\leavevmode\vtop{
		\baselineskip\z@skip \lineskiplimit-\maxdimen
		\ialign{\hfil$#1##$\hfil\cr#2\crcr}}}
\makeatother

\newcommand{\plangle}{\moverlay{(\cr<}}
\newcommand{\prangle}{\moverlay{)\cr>}}
\newcommand{\rx}{\C\plangle x \prangle}
\newcommand{\rtx}{\C\plangle \tilde{x} \prangle}

\newcommand{\nice}{unsignatured\xspace}
\newcommand{\epic}{epic\xspace}

\makeatletter
\newtheorem*{rep@thm}{\rep@title}
\newcommand{\newreptheorem}[2]{%
\newenvironment{rep#1}[1]{%
 \def\rep@title{#2 \ref{##1}}%
 \begin{rep@thm}}%
 {\end{rep@thm}}}
\makeatother

\newtheorem{thm}{Theorem}[section]
\newreptheorem{thm}{Theorem}
\newtheorem{lem}[thm]{Lemma}
\newtheorem{cor}[thm]{Corollary}
\newreptheorem{cor}{Corollary}
\newtheorem{prop}[thm]{Proposition}

\theoremstyle{definition}
\newtheorem{defn}[thm]{Definition}
\newtheorem{exa}[thm]{Example}
\theoremstyle{remark}
\newtheorem{rem}[thm]{Remark}

\makeatletter
\newcommand{\mycontentsbox}{%
	{\centerline{NOT FOR PUBLICATION}
		\addtolength{\parskip}{-2.3pt}
		\tableofcontents}}
\def\enddoc@text{\ifx\@empty\@translators \else\@settranslators\fi
	\ifx\@empty\addresses \else\@setaddresses\fi
	\newpage\mycontentsbox\newpage\printindex}
\makeatother

\numberwithin{equation}{section}

\title[Nullstellens{\"a}tze for the free algebra]{Factorization of noncommutative polynomials\\[1mm]  and Nullstellens{\"a}tze for the free algebra}

\author[J. W. Helton]{J. William Helton${}^1$}
\address{J. William Helton, Department of Mathematics, University of California San Diego}
\email{helton@math.ucsd.edu}
\thanks{${}^1$Supported by the NSF grant DMS 1500835.}

\author[I. Klep]{Igor Klep${}^2$}
\address{Igor Klep, Department of Mathematics, University of Ljubljana}
\email{igor.klep@fmf.uni-lj.si}
\thanks{${}^2$Supported by the Slovenian Research Agency grants J1-8132, N1-0057 and P1-0222, 
	and partially supported by
	the Marsden Fund Council of the Royal Society of New Zealand.}

\author[J. Vol\v{c}i\v{c}]{Jurij Vol\v{c}i\v{c}${}^3$}
\address{Jurij Vol\v{c}i\v{c}, Department of Mathematics, Texas A\&M University}
\email{volcic@math.tamu.edu}
\thanks{${}^3$Supported by the Deutsche Forschungsgemeinschaft (DFG) Grant No. SCHW 1723/1-1.}

\date{\today}

\begin{document}
	
\begin{abstract}
This article gives a class of  Nullstellens\"atze for noncommutative polynomials.
The singularity set of a noncommutative polynomial $f=f(x_1,\dots,x_g)$ is $\fl(f)=(\fl_n(f))_n$,
where $\fl_n(f)=\{X \in \mat{n}^g \colon \det f(X) = 0\}.$ 
The first main theorem of this article shows that 
the irreducible factors of $f$ are in a natural bijective correspondence with irreducible components of $\fl_n(f)$
for every sufficiently large $n$. 

With each polynomial $h$ in $x$ and $x^*$ one also associates 
its real singularity set $\fl^{\re}(h)=\{X\colon \det h(X,X^*)=0\}$.
A polynomial $f$  which depends on $x$ alone (no $x^*$ variables)  will be called analytic.
The main Nullstellensatz proved here is as follows:
for analytic $f$ but for $h$ dependent on possibly both  $x$ and $x^*$, the containment 
           $\fl(f) \subseteq \fl^{\re} (h)$
is equivalent to each factor of $f$  being ``stably associated'' to a factor of $h$ or of $h^*$.\looseness=-1

For perspective, classical Hilbert type Nullstellens\"atze typically apply only to analytic polynomials $f,h $, while real Nullstellens\"atze typically require adjusting the functions by sums of squares of polynomials (sos). Since the above  ``algebraic certificate'' does not involve a sos, it seems justified to think of this as the natural determinantal Hilbert Nullstellensatz. 
An earlier paper of the authors (Adv. Math. 331 (2018) 589--626) obtained such a theorem for special classes of analytic polynomials  $f$ and $h$. This paper requires few hypotheses and  hopefully brings this type of Nullstellensatz to near final form.\looseness=-1

 Finally, the paper gives a Nullstellensatz for zeros $\cV(f)=\{X\colon f(X,X^*)=0\}$ of a hermitian polynomial $f$, leading to a strong Positivstellensatz for quadratic free semialgebraic sets by the use of a slack variable.
\end{abstract}

\maketitle


\section{Introduction}

Hilbert's Nullstellensatz is a fundamental result in
classical algebraic geometry describing polynomials
vanishing on a complex algebraic variety. It has been generalized or extended to various settings, including many noncommutative ones. For instance,
Amitsur's Nullstellensatz \cite{Ami57} (see also \cite{Pro66}) describes noncommutative polynomials vanishing on the common vanishing set of a given finite set of polynomials in a full matrix algebra. 
The papers
\cite{BK09,KS14}
discuss vanishing traces of noncommutative polynomials, Cimpri\v c \cite{Cim19} 
considers Nullstellensatz questions involving polynomial partial differential operators, 
 Salomon,  Shalit and Shamovich
\cite{SSS18} give free noncommutative analytic Nullstellens\"atze, 
Reichstein and  Vonessen
\cite{RV07} develop a 
 framework for such questions in the context of rings with polynomial identities (see also \cite{vOV81}), etc.\looseness=-1

In a slightly different direction, the 
real Nullstellensatz and Positivstellensatz are pillars of real algebraic geometry \cite{BCR}.
These too have seen a plethora of noncommutative extensions.
For example, the articles \cite{CHMN13,Oza16,Scm09}
develop general frameworks, 
and these ideas yield various applications, see e.g. Positivstellens\"atze on quantum graphs \cite{NT15}, isometries \cite{HMP}, convex semialgebraic sets \cite{HKM12}, and algebraic approaches to Connes' embedding conjecture \cite{Oza13}.

In this article we prove new complex and real Nullstellens\"atze for the free algebra and provide a geometric interpretation for factorization of noncommutative polynomials \cite{Coh,BS,BHL}. 
The motivation for this work comes from the rapidly emerging areas of free analysis \cite{AM,KVV3,SSS18,DK} and free real algebraic geometry \cite{Oza13,NT15,HKMV} that study function theory on (semi)algebraic sets in the space of matrix tuples of all sizes.
For noncommutative polynomials $f$ there are three major types of zeros. These are
hard zeros $\{X \colon f(X)=0 \}$ \cite{Ami57,SSS18}, 
directional zeros
$\fl_{\rm dir}(f)=\{(X,v) \colon f(X)v =0 \}$ \cite{HM04,HMP07}, and determinantal zeros 
$\fl(f)=\{X\colon \det f(X)=0\}$ \cite{KV17,HKV}.
For each of these one hopes there will be a  Hilbert type Nullstellensatz (treating inclusion of zero sets) and a real Nullstellensatz (treating ``real'' points in zero sets). For directional zeros, reasonably satisfying theorems along these lines exist; see \cite{HMP07} for a Nullstellensatz and \cite{CHMN13} for a real Nullstellensatz. This article concerns determinantal zeros in Sections \ref{sec2} and \ref{sec3}, and hard zeros in Section \ref{sec4}.\looseness=-1

Our algebraic certificates are a counterpart to Hilbert's Nullstellensatz, 
and such theorems often carry direct applications to, say, semidefinite optimization and control theory \cite{SIG,BGM}. 
Furthermore, we hope these results may be of interest to researchers in noncommutative algebra, matrix theory or polynomial identities.

\subsection{Main results}

Let $x=(x_1,\dots,x_g)$ be freely noncommuting variables. Elements of
$$\ax{}^{\de\times\de}=\mat{\de}\otimes \ax$$
are {\bf noncommutative matrix  polynomials}. If $f=\sum_{w}f_ww\in\ax^{\de\times \de}$ and $X\in\mat{n}^g$, then $f(X)=\sum_{w}f_w\otimes w(X)\in\mat{\de n}$ denotes the evaluation of $f$ at $X$.

\subsubsection{Complex Nullstellens\"atze}

Our first main result gives a geometric interpretation of irreducibility in free algebras. We say that $f\in\ax^{\de\times\de}$ is {\bf full} \cite[Section 0.1]{Coh} if it cannot be factored as $f=f_1f_2$ for $f_1\in\ax^{\de\times \ve}$ and $f_2\in\ax^{\ve\times \de}$ with $\ve<\de$. In the special case $\delta=1$, a polynomial $f\in\ax$ is full if and only if $f\neq0$. Following \cite{Coh} we call $f$ an {\bf atom} if $f$ is not invertible in $\ax^{\de\times\de}$ and $f$ cannot be written as a product of non-invertible elements in $\ax^{\de\times\de}$. In particular, every atom is full. Every full matrix admits a complete factorization into atoms \cite[Proposition 3.2.9]{Coh}.

\begin{repthm}{t:irr}
A matrix polynomial $f$ is an atom if and only if
$\det f|_{\mat{n}^g}$
is an irreducible polynomial for all but finitely many $n\in\N$.
\end{repthm}

This theorem leads to a geometric description of factorization in free algebras. 
The geometric objects we use are free singularity sets. 
For $f\in\ax^{\de\times\de}$ let
$$
\fl(f) =\bigcup_n \fl_n(f),\qquad \fl_n(f) =\left\{X\in\mat{n}^g
\colon \det f(X)=0 \right\}
$$
be its {\bf free locus}. In \cite[Theorem 4.3]{HKV} it was shown that components of $\fl(f)$ correspond to a factorization of $f$ in $\ax^{\de\times\de}$ \emph{if $f(0)$ is an invertible matrix}. Theorem \ref{t:new} below,
 a free analog of Hilbert's Nullstellensatz,
 disposes of this assumption. Its proof is based on the aforementioned special case and novel techniques involving point-centered ampliations and representation theory.

To deal with the non-uniqueness of factorization in free algebras, Cohn introduced stable associativity \cite[Section 0.5]{Coh}.
Polynomials $f_1\in\ax^{\de_1\times\de_1}$ and $f_2\in\ax^{\de_2\times\de_2}$ are 
{\bf stably associated}  if there exist $\ve_1,\ve_2\in\N$ with $\de_1+\ve_1=\de_2+\ve_2$ and invertible $P,Q\in\ax^{(\de_1+\ve_1)\times (\de_1+\ve_1)}$ such that
$$f_2\oplus I_{e_2} = P(f_1\oplus I_{e_1})Q.$$
Actually, if $f_1$ and $f_2$ are stably associated, one can always choose $\ve_1=\de_2$ and $\ve_2=\de_1$ \cite[Theorem 0.5.3]{Coh}. In particular, for $f_1,f_2\in\ax$ stable associativity becomes a question about $2\times 2$ matrices over $\ax$. There is a straightforward procedure to generate all stably associated pairs, see \cite[Section 2.7]{Coh}.

\begin{repthm}{t:new}[Singul{\"a}rstellensatz]
Let $f_1,\dots,f_s,h$ be full matrix polynomials. Then $\bigcap_j\fl(f_j)\subseteq \fl(h)$ if and only if for some $1\le j\le s$, every atomic factor of $f_j$ is stably associated to a factor of $h$.
\end{repthm}

See Theorem \ref{t:new} and Proposition \ref{p:inter} for the proof.

\subsubsection{Real Nullstellens\"atze}
We next turn our attention to the ``real'' setting with an involution. 
Let $x^*=(x_1^*,\dots,x_g^*)$ be formal adjoints to $x$. The map $x_j\mapsto x_j^*$ extends to a unique involution $*$ on
$$\px{}^{\de\times\de}=\mat{\de}\otimes\px$$
restricting to the conjugate transpose on $\mat{\de}$. For $f\in\px^{\de\times\de}$ we define its {\bf real free locus} and the {\bf real free zero set}:
\begin{alignat*}{3}
\fl^{\re}(f)&=\bigcup_{n\in\N}\fl_n^{\re}(f),\qquad 
&\fl_n^{\re}(f) &= \left\{X\in \mat{n}^g\colon \det f(X,X^*)=0\right\},\\
\cV^{\re}(f)&=\bigcup_{n\in\N}\cV_n^{\re}(f),\qquad 
&\cV_n^{\re}(f) &= \left\{X\in \mat{n}^g\colon f(X,X^*)=0\right\}.
\end{alignat*}

A matrix polynomial $f$ depending on $x$ but not on $x^*$ is called {\bf analytic}. For such an $f$ we have
$\fl^{\re}(f)=\fl(f)$.

\begin{repthm}{t:an}[Analytic Singul{\"a}rstellensatz]
Let $f_1,\ldots,f_s$ be analytic atoms in $x$ and $h$ a full matrix polynomial. Then 
$\bigcap_j\fl^{\re}(f_j)\subseteq\fl^{\re}(h)$ if and only if 
there is $j$ such that
$f_j$ or $f_j^*$ is stably associated to a factor of $h$.
\end{repthm}

A straightforward extension of Theorem \ref{t:an}
where both $f,h$ are allowed to contain $x^*$ fails even if $f=f^*$ is hermitian, see Example \ref{e:sos}. A natural class of $f$ for which the conclusion does hold are \nice matrix polynomials $f$. A hermitian polynomial $f=f^*$ is {\bf \nice} if there exist $n\in\N$ and $X,Y\in\mat{n}^g$ such that $f(X,X^*),f(Y,Y^*)$ are invertible and have different signatures.

\begin{repthm}{t:herm}[Hermitian Singul{\"a}rstellensatz]
Let $h$ be a full matrix polynomial and let $f$ be an \nice atom. Then $\fl^{\re}(f)\subseteq\fl^{\re}(h)$ if and only if $f$ is stably associated to a factor of $h$.
\end{repthm}

A final main result is a Nullstellensatz for real free zero sets of polynomials with a distinguished quadratic term. As with the unsignatured hypothesis in Theorem \ref{t:herm} for real free loci, this was done under a certain definiteness assumption.

\begin{repthm}{t:null}
For a nonconstant hermitian $f\in \px$ assume $\{f\succ0\}\neq\emptyset$. Let $h\in\pxy$. Then
$\cV^{\re}(f-y^*y)\subseteq \cV^{\re}(h)$ if and only if $h\in(f-y^*y)$.
\end{repthm}

As a consequence we obtain a necessary and sufficient Positivstellensatz for hereditary quadratic polynomials, by adding a slack variable.

\begin{repcor}{c:posss}
Let $f\in\px$ be a nonconstant hermitian hereditary quadratic polynomial with $\{f\succ0\}\neq\emptyset$, and let $y$ be an auxiliary variable. If $h\in\px$, then $h|_{\{f\succeq0\}}\succeq0$ 
if and only if
$$h=f_0+\sum_j f_j^*f_j$$
for some $f_j\in \pxy$ with $f_0\in (f-y^*y)$.
\end{repcor}

While previously known necessary and sufficient Positivstellens\"atze on free semialgebraic sets do not require a slack variable, they only hold if the underlying free semialgebraic set is either convex with a nonempty interior \cite{HKM12} or given by quadratic equations, such as spherical isometries or tuples of unitaries \cite{KVV}. On the other hand, Corollary \ref{c:posss} is an example of a Positivstellensatz on a (possibly) non-convex free semialgebraic set with a nonempty interior.

\subsubsection{Linear Gleichstellensatz}
An affine matrix polynomial is traditionally called a linear pencil. It is indecomposable if it cannot be put in block triangular form with a left and right basis change, cf.~Definition \ref{d:ind}.
One can effectively apply the above Nullstellens\"atze to indecomposable linear pencils $L$, $M$ to get 
roughly: $M = P LQ$ for some $P, Q \in \GL_d(\C)$ if and  only if the the free loci of $L$ and $M$ coincide.
This is true for complex zeros (Theorem \ref{t:ind}), real zeros (Theorem \ref{c:penc0}) and in the context of the
 Hermitian Nullstellensatz, which requires extra conditions on zeros  (Theorem \ref{c:penc0}).

\subsubsection{Zero sets over the reals}

When dealing with real matrix polynomials, it suffices to consider only their evaluations at tuples of real matrices. Namely, for each $n\in\N$ we have a $*$-embedding $\iota:\mat{n}\hookrightarrow \opm_{2n}(\R)$ induced by $\alpha+\beta i\mapsto
\left(\begin{smallmatrix}\alpha & -\beta \\ \beta& \alpha\end{smallmatrix}\right)$. Note that $\iota(X)$ is unitarily equivalent to $X\oplus \overline{X}$ for each $X\in\mat{n}^g$, where $\overline{X}$ is the entry-wise complex conjugate of $X$. Hence if $f\in \R\!\Langle x,x^*\Rangle^{\de\times\de}$, then $f(X,X^*)$ is singular (resp. zero) if and only if $f(\iota(X),\iota(X)^\ti)$ is singular (resp. zero). Therefore one can replace free loci and free zero sets in the above theorems with their counterparts over $\R$ when applied to real polynomials, which is usually the preferred setting in control theory and optimization.

\subsection*{Acknowledgments}
The authors thank Scott McCullough for insightful discussions. The first two named authors thank the Mathematisches Forschungsinstitut Oberwolfach (MFO) for support through the ``Research in Pairs'' (RiP) program in 2017.


\section{Factorization in a free algebra and free loci}\label{sec2}

This section has two main results. 
In Theorem \ref{t:irr} we prove that a
matrix polynomial $f$ is an atom if and only if $\fl_n(f)$ is eventually a reduced irreducible hypersurface.
Theorem \ref{t:new} shows that  $f$ divides $h$ in the sense that every
atomic factor of $f$ is stably associated to a factor of $h$
if and only if $\fl(f)\subseteq\fl(h)$. 
These results  are far-reaching generalizations of \cite[Theorems A, B]{HKV} 
to matrix polynomials that eliminate
the assumption of $f(0)$ being invertible.
The proofs here rely on representation theory and point-centered ampliations, whereas
\cite{HKV} used invariant theory. For the sake of convenience and later sections we use $\C$ as the base field, but all proofs work for an arbitrary algebraically closed field of characteristic 0.

For $n\in\N$ let $\gX^n=(\gX_1^n,\dots,\gX_g^n)$ be a tuple of $n\times n$ generic matrices. That is,
$\gX_j^n=(\omega_{j\imath\jmath})_{\imath\jmath}$, where $\omega=(\omega_{j\imath\jmath}:\imath,\jmath,j)$ are independent commuting variables. We view the entries of $\gX^n$ as the coordinates of the affine space $\mat{n}^g$.

By \cite[Theorem 5.8.3]{Coh}, a full matrix $f\in\ax^{\de\times\de}$ is stably associated to a linear pencil $L=A_0+A_1x_1+\cdots A_gx_g$ of size $d$ such that
$$\begin{pmatrix}A_1&\cdots&A_g\end{pmatrix} \quad\text{and}\quad \begin{pmatrix}A_1\\\vdots\\A_g\end{pmatrix}$$
have full rank. We call such $L$ an {\bf \epic\footnote{
In \cite{Coh} such $L$ is called {\it monic}, which we avoid since monic pencils in control theory and convexity usually refer to pencils with $A_0=I$.	
} pencil}. We also say that $L$ is a {\bf linearization} of $f$. By the definition of stable associativity there is $\alpha\in\C\setminus\{0\}$ such that $\det f(\gX^n)=\alpha^n\det L(\gX^n)$ for all $n\in\N$. Also, $f$ is full if and only if $L$ is full by \cite[Theorem 7.5.13]{Coh}. Furthermore, $f$ is an atom if and only if $L$ is an atom by \cite[Proposition 0.5.2, Corollary 0.5.5 and Proposition 3.2.1]{Coh}.

\begin{exa}
If $f=x_1x_2-x_2x_1$, then one can check that the pencil
$$L=A_0+A_1x_1+A_2x_2=\begin{pmatrix}
-1 & 0 & x_2 \\
0 & -1 & x_1 \\
x_1 & -x_2 & 0
\end{pmatrix}$$
is a linearization of $f$. While $L$ is epic, no linear combination of $A_0,A_1,A_2$ is invertible, which corresponds to $f$ vanishing on $\C^2$.
\end{exa}

We  next record facts about full and invertible matrices over $\ax^{\de\times\de}$ that are scattered across the existing literature.
\begin{lem}\label{l:full}
For $f\in\ax^{\de\times\de}$  the following are equivalent:
\begin{enumerate}[\rm(i)]
\item $f$ is full;
\item there are $n\in\N$ and $X\in\mat{n}^g$ such that $\det f(X)\neq0$;
\item there exists $n_0\in\N$ such that $\det f(\gX^n)\neq0$ for every $n\ge n_0$.
\end{enumerate}
Furthermore, a full $f$ is not invertible in $\ax^{\de\times \de}$ if and only if there exists $n_0\in\N$ such that $\det f(\gX^n)$ is nonconstant for every $n\ge n_0$.
\end{lem}

\begin{proof}
By \cite[Corollary 7.5.14]{Coh}, $f$ is full if and only if $f$ is invertible over  the free skew field of noncommutative rational functions $\rx$ (see Subsection \ref{ss:prelim} for further information about this skew field). The equivalence of (i) and (ii) now follows from the construction of $\rx$ described in \cite[Section 2]{KVV2}. On the other hand, (iii) is equivalent to (ii) by linearization and \cite[Proposition 2.10]{DM}.

If $f$ is invertible in $\ax^{\de\times \de}$, then $\det f(\gX^n)\det f^{-1}(\gX^n)=1$ is a product of polynomials,  so $\det f(\gX^n)$ is a nonzero constant. If $f$ is not invertible, then either $f$ is not full or $f$ has an atomic factor $h$. If $h(0)=0$, then $\det h(\gX^n)$ is not constant for large enough $n$ by the first part since $h$ is full. If $h(0)\neq0$, then there exists $n_0\in\N$ such that $h(\gX^n)$ is nonconstant for all $n\ge n_0$ by \cite[Theorem 4.3]{HKV}.
\end{proof}

\subsection{Point-centered ampliations}

Fix $X\in\mat{n}^g$ and let
$$y=(y_{j\imath\jmath}\colon 1\le j\le g, 1\le \imath,\jmath\le n)$$
be $gn^2$ freely noncommuting variables. For $f\in\ax^{\de\times\de}$ let
$$f^X = f\left(
X_1+(y_{1\imath\jmath})_{\imath,\jmath},\dots,X_g+(y_{g\imath\jmath})_{\imath,\jmath}
\right)\in \ay{}^{{\de n}\times{\de n}}$$
be its {\bf point-centered ampliation} at $X$. In particular, if $L=A_0+\sum_{j>0}A_jx_j$ and $0_n\in\mat{n}^g$ is the zero tuple, then $L^{0_n}$ is up to a canonical shuffle equal to
\begin{equation}\label{e:30}
(I_n\otimes A_0)
+\sum_{j=1}^n\sum_{\imath,\jmath=1}^n(E_{\imath\jmath}\otimes A_j)y_{j\imath\jmath},
\end{equation}
where $\otimes$ is Kronecker's product and $E_{\imath\jmath}\in\mat{n}$ are the standard matrix units. 
For applications of related ideas to noncommutative rational functions see \cite{Vol18,PV}.

In this subsection we prove that point ampliations preserve atoms. First we require two technical lemmas.

\begin{lem}\label{l:sets}
For sets $S_1,\dots,S_n$ we have
$$\sum_{i=1}^n \left(|S_i|+\left|S_i\setminus\bigcup_{k\neq i}S_k \right|\right)
\ge 2\left|\bigcup_{i=1}^nS_i\right|.$$
\end{lem}

\begin{proof}
Follows by induction on $n$ using
\[|S_1\cap S_2|
+\left|S_1 \setminus \bigcup_{k\neq 1}S_k\right|
+\left|S_2 \setminus \bigcup_{k\neq 2}S_k\right| 
\ge
\left|(S_1\cup S_2) \setminus \bigcup_{k\neq 1,2}S_k\right|.
\qedhere\]
\end{proof}

\begin{lem}\label{l:t1}
Let $A_0,\dots,A_g\in\mat{d}$ and $n,d'd''\in\N$ satisfy $d'+d''=nd$. Assume there exist $P_i\in\C^{d'\times d}$, and $Q_i\in\C^{d\times d''}$ for $i=1,\dots,n$ such that
$$\begin{pmatrix} P_1 & \cdots & P_n\end{pmatrix}\qquad \text{and}\qquad
\begin{pmatrix} Q_1 \\ \vdots \\ Q_n\end{pmatrix}$$
have full rank and
\begin{equation}\label{e:31}
\begin{split}
\sum_iP_iA_0Q_i &= 0\\
P_iA_jQ_k &= 0\qquad \text{for all}\quad 1\le i,k\le n,\ j>0.
\end{split}
\end{equation}
Then there exist $U\in \C^{e'\times d}$, $V\in \C^{d\times e''}$ of full rank for some $e',e''\in\N$ satisfying $e'+e''=d$ such that
\begin{equation}\label{e:32}
UA_jV = 0 \qquad \text{for all}\quad j\ge0.
\end{equation}
\end{lem}

\begin{proof}
Since the statement is trivial for $n=1$, let $n\ge2$. By assumption we have $\C^{d'}=\sum_i\ran P_i$. For $1\le i\le n$ let $\widehat{P}_i\in\mat{d'}$ be the projection onto the complement of $\ran P_i\cap\sum_{k\neq i}\ran P_k$ in $\ran P_i$ along $\sum_{k\neq i}\ran P_k$. Analogously we define $\widehat{Q}_i\in\mat{d''}$ with respect to $\ran Q_i^\ti$.

By Lemma \ref{l:sets} we have
\begin{align*}
\sum_i(\rk P_i+\rk \widehat{P}_i)&\ge 2\dim\left(\sum_i\ran P_i\right)=2d',\\
\sum_i(\rk Q_i+\rk \widehat{Q}_i)&\ge 2\dim\left(\sum_i\ran Q_i^\ti\right)=2d''.
\end{align*}
Let $m_i=\rk \widehat{P}_i+\rk Q_i$ and $n_i=\rk P_i+\rk \widehat{Q}_i$. Then
$$\sum_im_i+\sum_in_i = \sum_i(\rk P_i+\rk \widehat{P}_i+\rk Q_i+\rk \widehat{Q}_i)\ge 2d'+2d''=2nd.$$
Therefore there exists $1\le i\le n$ such that $m_i\ge d$ or $n_i\ge d$.

Suppose $m_1\ge d$ (the other cases are treated analogously). Then \eqref{e:31} implies
$${\widehat{P_1}P_1}A_jQ_1=0 \qquad \text{for all}\quad j\ge0.$$
Since $\rk (\widehat{P_1}P_1)=\rk \widehat{P_1}$, there exist $e'\le \rk \widehat{P_1}$ and $e''\le \rk Q_1$ satisfying $e'+e''=d$. By choosing $e'$ linearly independent rows of $\widehat{P_1}P_1$ and $e''$ linearly independent columns of $Q_1$ we obtain $U\in \C^{e'\times d}$, $V\in \C^{d\times e''}$ of full rank such that \eqref{e:32} holds.
\end{proof}

\begin{defn}\label{d:ind}
A linear pencil $L$ of size $d$ is {\bf indecomposable} if there are no $P,Q\in\GL_d(\C)$ such that
$$PLQ=\begin{pmatrix}\star& 0 \\ \star & \star\end{pmatrix},$$
where the zero block is of size $d'\times d''$ with $d'+d''=d$.
\end{defn}

\begin{rem}
When restricted to monic pencils ($L(0)=I$), Definition \ref{d:ind} coincides with the one in \cite{HKMV} (while in \cite{HKV} such pencils were called irreducible): namely, a monic pencil is indecomposable if its coefficients admit no nontrivial common invariant subspace (or equivalently, generate the full matrix algebra). If $L(0)$ is invertible, the new definition can be reduced to the old one since $L$ is indecomposable if and only if $L(0)^{-1}L$ is indecomposable, and the latter is a monic pencil.

We will frequently use \cite[Theorem 5.8.8]{Coh} stating that an epic pencil is an atom if and only if it is indecomposable.
\end{rem}

\begin{prop}\label{p:t2}
Let $f$ be a matrix polynomial and $X\in \mat{n}^g$. If $f$ is an atom, then $f^X$ is an atom.
\end{prop}

\begin{proof}
Let $L$ be an \epic pencil that is stably associated to $f$. Then $L^X$ is stably associated to $f^X$ for every $X\in\mat{n}^g$ and $n\in\N$. Since stable associativity preserves atoms, $L$ is an atom and it suffices to show that every point-centered ampliation of $L$ is an atom. Furthermore, an affine change of variables preserves atoms, so it suffices to consider point-centered ampliations at $X=0_n$ for $n\in\N$, where $0_n$ is the zero tuple in $\mat{n}^g$.

Let $L=A_0+\sum_jA_jx_j$ be an \epic pencil of size $d$. Then $L^{0_n}$ is an \epic pencil of size $nd$ and up to a basis change equal to \eqref{e:30}. Suppose that $L^{0_n}$ is not an atom. Since it is epic, it is not indecomposable by \cite[Theorem 5.8.8]{Coh}, so there exist $d',d''\in\N$ satisfying $d'+d''=nd$ and $P_0,Q_0\in\GL_{nd}(\C)$ such that
\begin{equation}\label{e:33}
P_0L^{0_n}Q_0=\begin{pmatrix}\star& 0 \\ \star & \star\end{pmatrix},
\end{equation}
where the zero block is of size $d'\times d''$. Let $(P_1\ \cdots\ P_n)$ be the first $d'\times (nd)$ block row of $P_0$, and let $(Q_1^\ti \ \cdots\  Q_n^\ti)^\ti$ be the last $(nd)\times d''$ block column of $Q_0$. Then the assumptions of Lemma \ref{l:t1} are satisfied by \eqref{e:33} and \eqref{e:30}. So there exist $P,Q\in \GL_d(\C)$ such that
$$PLQ=\begin{pmatrix}\star& 0 \\ \star & \star\end{pmatrix},$$
where the zero block is of size $e'\times e''$ and $e'+e''=d$. Therefore $L$ is not indecomposable and hence not an atom by \cite[Theorem 5.8.8]{Coh}.
\end{proof}

\subsection{Irreducibility}

In this section we prove the main irreducibility result on free loci (Theorem \ref{t:irr}). 
We start with an observation about the degrees of determinants of a matrix polynomial; for a related rank-stabilizing result see \cite[Theorem 1.8]{DM}.

\begin{lem}\label{l:deg}
Let $f$ be a full non-invertible matrix polynomial. Then there exist $n_0,d\in\N$ such that $\deg\det f(\gX^n)=dn$ for all $n\ge n_0$.
\end{lem}

\begin{proof}
For $n\in\N$ let $d_n = \deg\det f(\gX^n)$. Since $f(\gX^{n'}\oplus \gX^{n''})=f(\gX^{n'})\oplus f(\gX^{n''})$,
we have
\begin{equation}\label{e:ineq}
d_{n'}+d_{n''}\le d_{n'+n''}
\end{equation}
for all $n',n''\in\N$. By Lemma \ref{l:full} there exist $m\in\N$ and $X^i\in\mat{m+i}^g$ for $i=0,1$ such that $\det f^{X^i}(0)=\det f(X^i)\neq0$. Then there exist $k_0,c_i\in\N$ such that
\begin{equation}\label{e:dkm}
d_{k(m+i)}=\deg \det f^{X^i}(\gX^k)=c_i k
\end{equation}
for all $k\ge k_0$ by \cite[Lemma 3.1]{HKV} and linearization.
Choosing $k=m(m+1)k_0$ yields $mc_1=(m+1)c_0$. Furthermore, for all $n\in\N$ we have
$$mk_0 d_n\le d_{mk_0n}=c_0k_0n$$
by \eqref{e:ineq} and \eqref{e:dkm}. On the other hand, every $n\ge m(2k_0+m-1)$ can be written as $n=k_1m+k_2(m+1)$ for some $k_1,k_2\ge k_0$, so \eqref{e:ineq} implies
\begin{align*}
(m+1)c_0n
&= (m+1)c_0k_1m+mc_1k_2(m+1) \\
&= m(m+1)(c_0k_1+c_1k_2)\\
&= m(m+1)(d_{k_1m}+d_{k_2(m+1)})\\
&\le m(m+1)d_n.
\end{align*}
Hence $d_n=\frac{c_0}{m}n$ for all $n\ge m(2k_0+m-1)$, and consequently $\frac{c_0}{m} \in\N$.
\end{proof}

The next irreducibility theorem is one of the central results in this paper. It is the pillar of the Hilbert type Nullstellens\"atze that follow in Sections \ref{sec2} and \ref{sec3}.

\begin{thm}\label{t:irr}
Let $f$ be a matrix polynomial. Then $f$ is an atom if and only if there is $n_0\in\N$ such that $\det f(\gX^n)$ is an irreducible polynomial for every $n\ge n_0$.
\end{thm}

\begin{proof}
Assume $f$ is not an atom. If $f$ is not full, then $\det f(\gX^n)=0$ for all $n$. If $f$ is invertible over $\ax$, then $\det f(\gX^n)$ is a nonzero constant for every $n$. If $f$ is full and not invertible, then $f=f_1f_2$ for some full non-invertible matrix polynomials $f_1,f_2$. Then $\det f(\gX^n)=\det f_1(\gX^n)\det f_2(\gX^n)$ is a proper factorization for all sufficiently large $n$ by Lemma \ref{l:full}. This proves one implication of the theorem.

Now let $f$ be an atom. By Lemma \ref{l:deg} there exist $m,d\in\N$ such that $\deg\det f(\gX^n)=dn$ for all $n\ge m$. In particular, there is $X\in\mat{m}^g$ such that $\det f(X)\neq0$. Then $f^X$ is an atom by Proposition \ref{p:t2}. Since $\det f^X(0) = \det f(X)\neq0$, there exists $k_0\in\N$ such $\det f^X(\gX^k)$ is an irreducible polynomial for every $k\ge k_0$ by \cite[Theorem 4.3]{HKV}. Since an affine change of variables does not affect irreducibility, $\det f^{0_m}(\gX^k)$ is also irreducible for every $k\ge k_0$. By the definition of $f^{0_m}$ we then conclude that $\det f(\gX^{km})$ is irreducible for all $k\ge k_0$.

Now let $n\ge (\max\{2d,k_0\}+1)m$. Then $n=km+r$ for $k\ge \max\{2d,k_0\}$ and $m\le r<2m$. Suppose that $\det f(\gX^n)=pq$ for $p,q\in\C[\gx]$. By the choice of $m$ in the previous paragraph there is $X\in\mat{r}^g$ such that $\det f(X)\neq0$. Since the polynomial $\det f (X\oplus \gX^{km})=\det f (X)\det f(\gX^{km})$ is irreducible, we can without loss of generality assume that $p$ evaluated at $X\oplus \gX^{km}$ equals $1$. Thus $\deg q\ge dkm$, and so $\deg p\le dn-dkm=dr<2dm$.

By \cite[Lemma 2.1]{HKV}, $p$ is given by a pure trace polynomial; that is, there is a formal polynomial $t$ in traces of words over $x$,
\begin{equation}\label{e:tr}
t=\sum_{i=1}^k \alpha_i \prod_{j=1}^{\ell_i}\tr(w_{ij}),
\end{equation}
such that the degree of $t$ equals the degree of $p$ and $p=t(\gX^{n})$. We also consider  another pure trace polynomial
\begin{equation}\label{e:tr1}
u=\sum_{i=1}^k \alpha_i \prod_{j=1}^{\ell_i}\big(\tr(w_{ij})+\tr(w_{ij}(X))\big).
\end{equation}
Note that $\deg u=\deg t$ and $u(\gX^{(km)})=p(X\oplus \gX^{km})=1$. Hence $u-1$ is a pure trace identity for $(km)\times (km)$ matrices which has degree $\deg p <2dm\le km+1$. Therefore $u-1$ is the zero polynomial by \cite[Theorem 4.5]{Pro}. Since $\deg u=\deg t$, the trace polynomial $t$ is also zero, so $p$ is constant. Hence $\det f(\gX^n)$ is irreducible for every $n\ge (\max\{2d,k_0\}+1)m$.
\end{proof}

\subsection{Complex Nullstellensatz}

In this subsection we show that indecomposable pencils are determined by their free loci (Theorem \ref{t:ind}), which then leads to the geometric reformulation of factorization in the free algebra (Theorem \ref{t:new}).

Consider the actions of $\SL_d(\C)\times \SL_d(\C)$ and $\GL_d(\C)\times \GL_d(\C)$ on $\mat{d}^{g+1}$ given by $(P,Q)\cdot X=PXQ^{-1}$. The following fact is probably well-known to specialists in invariant theory. We include a proof for the sake of completeness.

\begin{lem}\label{l:closed}
Let $L=A_0+\sum_{j>0}A_jx_j$ be of size $d$. If $L$ is indecomposable, then the $\SL_d(\C)\times \SL_d(\C)$-orbit of $A=(A_0,\dots,A_g)$ in $\mat{d}^{g+1}$ is Zariski closed, and the $\GL_d(\C)\times \GL_d(\C)$-stabilizer of $A$ is $\{(\alpha I,\alpha^{-1}I)\colon \alpha\in\C\setminus\{0\}\}$.
\end{lem}

\begin{proof}
Since $L$ is indecomposable, we have
$$\dim\left(\sum_{j=0}^g A_jV\right)-\dim V>0$$
for every proper subspace $V$ in $\C^d$. If we view $A$ as a $(d,d)$-dimensional representation
\[
\begin{tikzcd}
\C^d \arrow[r,bend left=50,"A_0"]
\arrow[r,bend left=20,"A_1"] 
\arrow[r,bend right=2,phantom,swap,"\vdots"]
\arrow[r,bend right=40,swap,"A_g"] & \C^d
\end{tikzcd}
\]
of the $(g+1)$-Kronecker quiver, then the previous observation implies that $A$ is $(1,-1)$-stable according to \cite[Definition 1.1 and Section 3]{Kin}. By \cite[Proposition 3.1]{Kin}, stability of $A$ as a quiver representation is equivalent to stability of $A$ as a point in $\mat{d}^{g+1}$ with the action of $\GL_d(\C)\times \GL_d(\C)$ according to \cite[Definition 2.1]{Kin}. Note that $\Delta=\{(\alpha I,\alpha^{-1}I)\colon \alpha\in\C\setminus\{0\}\}$ are precisely elements in $\GL_d(\C)\times \GL_d(\C)$ that act trivially on $\mat{d}^{g+1}$. Therefore stability of $A$ implies $$(\GL_d(\C)\times \GL_d(\C))\cdot A = \dim (\GL_d(\C)\times \GL_d(\C))/\Delta,$$
so the stabilizer of $A$ equals $\Delta$. Furthermore, since $\SL_d(\C)\times \SL_d(\C)$ is the commutator subgroup of $\GL_d(\C)\times \GL_d(\C)$, the $\SL_d(\C)\times \SL_d(\C)$-orbit of $A$ in $\mat{d}^{g+1}$ is Zariski closed by \cite[Theorem 1(i)$\Rightarrow$(iii)]{Shm}.
\end{proof}

\begin{thm}[Linear Gleichstellensatz]\label{t:ind}
Let $L$ and $M$ be indecomposable linear pencils of sizes $d$ and $e$, respectively. Then $\fl(L)=\fl(M)$ if and only if $d=e$ and $M=PLQ$ for some $P,Q\in\GL_d(\C)$.
\end{thm}

\begin{proof}
By Theorem \ref{t:irr} there exists $n_0\in\N$ such that $\det L(\gX^n)$ and $\det M(\gX^n)$ are irreducible for all $n\ge n_0$. By $\fl(L)=\fl(M)$ and irreducibility we see that $\det L(\gX^n)$ and $\det M(\gX^n)$ are equal up to a multiplicative constant for every $n\ge n_0$. Thus there exists $\alpha\in\C\setminus\{0\}$ such that $\det L(\gX^n)=\alpha^n\det M(\gX^n)$ for all $n\in \N$. After multiplying $L$ by $\alpha^{-1}$ we can therefore assume that
\begin{equation}\label{e:331}
\det L(\gX^n)=\det M(\gX^n) \qquad \text{for all }n\in\N.
\end{equation}

Let $L=A_0+\sum_{j>0}A_jx_j$ and $M=B_0+\sum_{j>0}B_jx_j$. Then
\begin{equation}\label{e:34}
\det\left(A_0\otimes Y+\sum_{j>0}A_j\otimes X_j\right)=
(\det Y)^d\det L(Y^{-1}X)
\end{equation}
for every $X\in\mat{n}^g$ and $Y\in\GL_n(\C)$, and similarly for $M$. By \eqref{e:331}, the polynomials
\begin{equation}\label{e:35}
(\det Y)^d\det L(Y^{-1}X),\qquad (\det Y)^e\det M(Y^{-1}X)
\end{equation}
agree up to a factor of $\det Y$. However, since $L$ is indecomposable, the left-hand side of \eqref{e:34} is an irreducible polynomial in $X$ for large enough $n$. Analogous conclusion holds for $M$, so the polynomials in \eqref{e:35} are equal, and thus $d=e$. Consequently, \eqref{e:331} and \eqref{e:34} imply
\begin{equation}\label{e:36}
\forall n\in\N,\ \forall X\in\mat{n}^{g+1}\colon\qquad
\det\left(\sum_{j=0}^gA_j\otimes X_j\right)=
\det\left(\sum_{j=0}^gB_j\otimes X_j\right).
\end{equation}

Let us view $A=(A_0,\dots,A_g)$ and $B=(B_0,\dots,B_g)$ as elements of $\mat{d}^{g+1}$ with the action of $\SL_d(\C)\times \SL_d(\C)$. Then $p(A)=p(B)$ for every $\SL_d(\C)\times \SL_d(\C)$-invariant polynomial function $p:\mat{d}^{g+1}\to \C$ by \eqref{e:36} and Theorem \cite[Theorem 2.3]{SvdB} or \cite[Theorem 1.4]{DM}. Therefore the Zariski closures of $\SL_d(\C)\times \SL_d(\C)$-orbits of $A$ and $B$ coincide. But the orbits of $A$ and $B$ are closed by Lemma \ref{l:closed}, so $A$ and $B$ lie in the same orbit.
\end{proof}

\begin{thm}[Singul{\"a}rstellensatz]\label{t:new}
Let $f$ and $h$ be full matrix polynomials. Then $\fl(f)\subseteq \fl(h)$ if and only if every atomic factor of $f$ is stably associated to a factor of $h$.\looseness=-1
\end{thm}

\begin{proof}
It suffices to assume that $f$ is an atom. Let $h=h_1\dots h_\ell$ be a factorization of $h$ into atoms.

$(\Leftarrow)$ If $f$ is stably associated to $h_i$, then there is $\alpha\neq0$ such that $\det f(\gX^n)=\alpha^n \det h_i(\gX^n)$ for all $n$, so $\fl(f)=\fl(h_i)\subseteq \fl(h)$.

$(\Rightarrow)$ Since $h_1,\dots,h_\ell$ are atoms, there exists $n_0\in \N$ such that $\fl_n(f)$ and $\fl_n(h_i)$ for $1\le i\le\ell$ are irreducible algebraic sets for every $n\ge n_0$. Since
$$\fl(f)\subseteq\fl(h)=\fl(h_1)\cup\cdots\cup\fl(h_\ell),$$
we conclude that for every $n\ge n_0$, $\fl_n(f)=\fl_n(h_i)$ for some $i$. Hence there exists $i$ such that $\fl_n(f)=\fl_n(h_i)$ for infinitely many $n$. Since $\fl_k(p)$ embeds into $\fl_{km}(p)$ via $X\mapsto X^{\oplus m}$ for every $k,m\in\N$ and any matrix polynomial $p$, we have $\fl_n(f)=\fl_n(h_i)$ for all $n$. Let $L$ and $M$ be the \epic pencils that are stably associated to $f$ and $h_i$. Then $L,M$ are atoms and thus indecomposable by \cite[Theorem 5.8.8]{Coh}. Since $\fl(L)=\fl(M)$, $L$ and $M$ are stably associated by Theorem \ref{t:ind}. Therefore $f$ and $h_i$ are stably associated.
\end{proof}

\begin{cor}\label{c:obvious}
Let $f$ be an atom and $h$ a full matrix polynomial. Then $\fl(f)= \fl(h)$ if and only if $h$ is a product of matrix polynomials each of which
is stably associated to  $f$.\looseness=-1
\end{cor}

Finally, let us record the following observation about free loci, which implies the introductory version of Theorem \ref{t:new} above, and will be used several times later in the text. While we usually think of free loci as analogs of hypersurfaces, their intersections do not behave as lower-dimensional varieties.

\begin{prop}\label{p:inter}
Let $f_1,\dots,f_s,h$ be matrix polynomials. If $\bigcap_j\fl(f_j)\subseteq \fl(h)$, then there exists $j\ge1$ such that $\fl(f_j)\subseteq \fl(h)$.
\end{prop}

\begin{proof}
Suppose $\fl(f_j)\nsubseteq \fl(h)$ for $j\ge2$. Hence there exist matrix tuples $X^2,\dots,X^s$ such that for $j\ge2$,
$$\det f_j(X^j)=0 \qquad \text{and}\qquad \det h(X^j)\neq0.$$
If $\det f_1(X)=0$ for some $X$, then
$$X\oplus \bigoplus_{j\ge 2}X^j \in \bigcap_{j=1}^s\fl(f_j)$$
and so $\det h(X)=0$. Therefore $\fl(f_1)\subseteq \fl(h)$.
\end{proof}


\section{Real Nullstellens\"atze}\label{sec3}

In this section we prove two new real Nullstellens\"atze for the free algebra. In Theorem \ref{t:an} we give a
geometric condition for an analytic (no $x^*$ variables) matrix polynomial $f$ to be a factor of an arbitrary matrix polynomial $h$. This result is under a natural assumption
extended to arbitrary $f$ in Theorem \ref{t:herm}. The proofs rely on preceding results in this paper and real algebraic geometry applied to the real structure on matrix tuples.

\subsection{Real structure}\label{ss:real}

For $f\in\px^{\de\times\de}$ and $(X,Y)\in\mat{n}^g\times\mat{n}^g$ let $f(X,Y)$ denote the involution-free evaluation of $f$ at $(X,Y)$ given by $x_j\mapsto X_j$ and $x_j^*\mapsto Y_j$, and let $f(X,X^*)$ denote the $*$-evaluation at $X$, where $X^*=(X_1^*,\dots,X_g^*)$.

Fix $n\in\N$. The map
\begin{equation}
\JJ:\mat{n}^g\times\mat{n}^g\to \mat{n}^g\times\mat{n}^g, \qquad (X,Y)\mapsto (Y^*,X^*)
\end{equation}
is conjugate-linear and involutive. Thus $\JJ$ is a {\bf real structure} on the complex space $\mat{n}^g\times\mat{n}^g$. Let $\cW$ be a complex algebraic set in $\mat{n}^g\times\mat{n}^g$. If $\JJ$ preserves $\cW$, then let $\cW^{\re}$ denote the set of points in $\cW$ fixed by $\JJ$,
$$\cW^{\re}=\cW\cap \bigcup_{n\in\N}\{(X,X^*)\colon X\in\mat{n}^g \}.$$
Then $\cW^{\re}$ is a real algebraic set, also called the set of {\bf real points} on $\cW$.

\begin{prop}\label{p:clas}
	Let $\cW$ be a complex algebraic set in $\mat{n}^g\times\mat{n}^g$.
	If $\cW$ is irreducible, $\cJ$ preserves $\cW$ and $\cW^{\re}$ contains a smooth point of $\cW$, then $\cW^{\re}$ is Zariski dense in $\cW$.
\end{prop}

\begin{proof}
	This is a special case of a more general statement about real points on a complex variety with a real structure, see e.g.~\cite[Lemma 1.5]{Be} or \cite[Theorem 4.9]{DE}.
\end{proof}

Recall the definition of the real free locus of $f\in\px^{\de\times\de}$ from the introduction. To derive results about $\fl^{\re}(f)$, we consider the (non-real) free locus of $f$ throughout this section in an involution-free way; that is, we ``forget'' the involutive relation between the variables $x$ and $x^*$, and thus $$\fl_n(f)=\{(X,Y)\in\mat{n}^g\times\mat{n}^g:\det f(X,Y)=0\}$$
as a complex algebraic set in $\mat{n}^g\times\mat{n}^g$. If $f=f^*$ is hermitian, then $\fl_n(f)$ is preserved by the real structure $\cJ$, and the real points of the free locus of $f$ are related to the real free locus of $f$ as follows: 
$$\fl_n(f)^{\re}=\left\{(X,X^*)\in\mat{n}^g\times \mat{n}^g\colon X\in \fl_n^{\re}(f)\right\}.$$

\subsection{Analytic Nullstellensatz}

Let $f\in\ax^{\de\times\de}\subset\px^{\de\times\de}$; such polynomials are called {\bf analytic}. Although $f$ contains no $x^*$, it is convenient to view it as a matrix over $\px$, and thus $\fl(f)=\{(X,Y):\det f(X)=0\}$. Observe that the real structure $\JJ$ on $\mat{n}^g\times\mat{n}^g$ preserves
$$\fl_n(f)\cap\fl_n(f^*)
=\left\{(X,Y)\in\mat{n}^g\times\mat{n}^g\colon \det f(X)=0=\det f^*(Y)\right\}$$
since $f(X^*)=f^*(X)^*$. The set of real points of this algebraic set is then
$$(\fl_n(f)\cap\fl_n(f^*))^{\re}
=\left\{(X,X^*)\in\mat{n}^g\times \mat{n}^g\colon X\in \fl_n^{\re}(f)\right\}.$$

\begin{prop}\label{p:an}
Let $f\in\ax^{\de\times\de}$ be an atom. There exists $n_0$ such that for every $n\ge n_0$, 
we have that
$(\fl_n(f)\cap\fl_n(f^*))^{\re}$ is Zariski dense in $\fl_n(f)\cap\fl_n(f^*)$.
\end{prop}

The proof uses smooth points, hence involves derivatives and their properties. 
Let $\gY^n=(\gY_1^n,\dots,\gY_g^n)$ be another tuple of $n\times n$ generic matrices. We view the entries of $(\gX^n,\gY^n)$ as the coordinates of the affine space $\mat{n}^g\times\mat{n}^g$.

\begin{lem}\label{l:der}
For all $X\in\mat{n}^g$,
$$\left.\left(\frac{\partial}{\partial \omega_{j\imath\jmath}} \det f(\gX^n)\right)\right|_{\gX^n= X} =
\overline{
	\left.\left(\frac{\partial}{\partial \upsilon_{j\jmath\imath}} \det f^*(\gY^n)\right)\right|_{\gY^n= X^*}
}.$$
\end{lem}

\begin{proof}
A consequence of the identity $\det f(X)=\overline{\det f^*(X^*)}$.
\end{proof}

\begin{proof}[Proof of Proposition \ref{p:an}]
By Theorem \ref{t:irr} there exists $n_0$ such that $\det f(\gX^n)$ is irreducible for all $n\ge n_0$. Now fix $n\ge n_0$. Since $\det f(\gX^n)$ is irreducible, there exists $X\in\mat{n}^g$ such that
\begin{equation}\label{e:sm1}
\left.\left(\grad_{\gX^n} \det f(\gX^n)\right)\right|_{\gX^n= X}\neq0,
\end{equation}
where $\grad_{\gX^n}$ denotes the gradient with respect to the $gn^2$ variables $\omega_{j\imath\jmath}$. By Lemma \ref{l:der} we also have
\begin{equation}\label{e:sm2}
\left.\left(\grad_{\gY^n} \det f^*(\gY^n)\right)\right|_{\gY^n= X^*}\neq0.
\end{equation}
The algebraic set $\fl_n(f)\cap\fl_n(f^*)$ is defined by $\det f(\gX^n)=0$ and $\det f^*(\gY^n)=0$, and the $2\times (gn^2+gn^2)$ Jacobian matrix of this pair has the form
$$\operatorname{J}(\gX^n,\gY^n)=\begin{pmatrix}
\grad_{\gX^n} \det f(\gX^n) & 0 \\
0 & \grad_{\gY^n} \det f^*(\gY^n)
\end{pmatrix}.$$
Then $\operatorname{J}(X,X^*)$ has rank 2 by \eqref{e:sm1} and \eqref{e:sm2}, so $(X,X^*)$ is a smooth point of $\fl_n(f)\cap\fl_n(f^*)$.
Finally, $\fl_n(f)\cap\fl_n(f^*)$ is irreducible since it is a product of two irreducible hypersurfaces in $\mat{n}^g$,
$$\fl_n(f)\cap\fl_n(f^*)=\{X\colon\det f(X)=0\}\times \{Y\colon\det f^*(Y)=0\}.$$
The statement then follows by Proposition \ref{p:clas}.
\end{proof}

\begin{thm}[Analytic Singul{\"a}rstellensatz]\label{t:an}
Let $f\in \ax^{\de\times\de}$ be an atom and $h\in\px^{\ve\times\ve}$ a full matrix. Then $\fl^{\re}(f)\subseteq\fl^{\re}(h)$ if and only if $f$ or $f^*$ is stably associated to a factor of $h$.
\end{thm}

\begin{proof}
Only $(\Rightarrow)$ is non-trivial. If $\fl^{\re}(f)\subseteq\fl^{\re}(h)$, then $(\fl_n(f)\cap\fl_n(f^*))^{\re}\subseteq \fl_n(h)$ for all $n$. By Proposition \ref{p:an}, $(\fl_n(f)\cap\fl_n(f^*))^{\re}$ is Zariski dense in $\fl_n(f)\cap\fl_n(f^*)$ for all $n$ large enough. Therefore $\fl_n(f)\cap\fl_n(f^*)\subseteq \fl_n(h)$ for all $n$ large enough, and consequently for all $n$. Therefore $\fl(f)\subseteq \fl(h)$ or $\fl(f^*)\subseteq \fl(h)$ by Proposition \ref{p:inter}, so the conclusion follows by Theorem \ref{t:new}.
\end{proof}

As a corollary we obtain the following somewhat unexpected statement.

\begin{cor}\label{c:an}
Let $h\in\px^{\ve\times\ve}$ and $f_j\in \ax^{\de_j\times\de_j}$ for $j=1,\dots,s$ be full matrices. Then 
\begin{equation}\label{e:re11}
\fl^{\re}(f_1)\cap\cdots\cap\fl^{\re}(f_s) \subseteq\fl^{\re}(h)
\end{equation}
if and only if there exists $j$ such that each atomic factor of $f_j$ is stably associated to a factor of either $h$ or $h^*$.
\end{cor}

\begin{proof}
If \eqref{e:re11} holds, then as in the proof of Proposition \ref{p:inter} we see that $\fl^{\re}(f_j)\subseteq \fl^{\re}(h)$ for some $j$. The rest is an immediate consequence of Theorem \ref{t:an}.
\end{proof}

Restricting Theorem \ref{t:an} to linear pencils yields the following Gleichstellensatz.

\begin{cor}\label{c:penc0}
Let $L,M$ be indecomposable linear pencils. If $L$ is analytic, then $\fl^{\re}(L)=\fl^{\re}(M)$ if and only if $L,M$ are of the same size $d$ and there exist $P,Q\in\GL_d(\C)$ such that $M=PLQ$ or $M=PL^*Q$.
\end{cor}

\begin{proof}
Combine Theorems \ref{t:ind} and \ref{t:an}.
\end{proof}

\subsection{Hermitian Nullstellensatz}

In real algebraic geometry, an ideal $J\subset\R[\gx]$ is called real if $J$ consists precisely of polynomials vanishing on the real zero set of $J$. Theorem \ref{t:herm} below is inspired by the characterization of real principal ideals \cite[Theorem 4.5.1]{BCR}. Namely, if $p\in\R[\gx]$ is irreducible, then $(p)$ is a real ideal if and only if $p$ changes sign. Recall that a hermitian matrix polynomial $f=f^*$ is called \nice if there exist $n\in\N$ and $X,Y\in\mat{n}^g$ such that $f(X,X^*),f(Y,Y^*)$ are invertible and have different signatures.

\begin{rem}
If $f(X,X^*)$ is positive (resp. negative) definite for some $X\in\mat{n}^g$, then $f$ is \nice if and only if $f$ (resp. $-f$) is not a sum of hermitian squares. There are also \nice polynomials that are never definite, for instance $f=x_1x_1^*-x_1^*x_1$ (because its trace is constantly 0). Another example of a non-\nice atom is
\begin{equation}\label{e:exa}
f=\begin{pmatrix}1+xx^* & x \\ x^* & -1-x^*x\end{pmatrix}.
\end{equation}
\end{rem}

\begin{prop}\label{p:nice}
Let $f$ be a hermitian polynomial, $n\in\N$ and $X,Y\in\mat{n}^g$ such that $f(X,X^*)$ and $f(Y,Y^*)$ are invertible with different signatures. If $\fl_n(f)$ is irreducible, then $\fl_n(f)^{\re}$ is Zariski dense in $\fl_n(f)$.
\end{prop}

\begin{proof}
As $\mat{n}^g\times\mat{n}^g$ is endowed with the real structure $\JJ$, we can view
$$\bA=\{(Z,Z^*)\colon Z\in \mat{n}^g \}$$
as the corresponding real affine space. There exists $\eta>0$ such that for every $Z\in\mat{n}^g$ with $\|Z-Y\|<\eta$, 
the matrix
$f(Z,Z^*)$ is invertible with the same signature as $f(Y,Y^*)$. Let $D$ be an open ball of radius $\eta$ about $Y$, intersected by the affine subspace through $Y$ that is perpendicular to $X-Y$, i.e.,
$$D=\left\{(Z,Z^*)\in\mat{n}^g\colon \|Z-Y\|<\eta \text{ and } \sum_j\tr((Z_j-Y_j)^*(X_j-Y_j))=0\right\}.$$
Then $D$ is a semialgebraic set in $\bA$ of (real) dimension $gn^2-1$. Let $\cC$ be the convex hull of $\{(X,X^*)\}\cup D$. If $(Z,Z^*)\in D$, then $f(Z,Z^*)$ and $f(X,X^*)$ have different signatures, so $\fl_n(f)^{\re}$ intersects the interior of the line segment between $(X,X^*)$ and $(Z,Z^*)$. Moreover, by the choice of $D$, every line through $(X,X^*)$ intersects $D$ at most once. Therefore we have a surjective map
$$\varphi: \fl_n(f)^{\re}\cap \cC\to D$$
given by projections onto $D$ along the lines through $(X,X^*)$. Then $\varphi$ is clearly semialgebraic, so $\dim(\fl_n(f)^{\re}\cap \cC)=gn^2-1$ by \cite[Theorem 2.8.8]{BCR}. Therefore its Zariski closure in $\mat{n}^g\times\mat{n}^g$ (with real structure $\cJ$) is a hypersurface by \cite[Proposition 2.8.2]{BCR}. Since the latter is contained in the irreducible $\fl_n(f)$, we conclude that $\fl_n(f)^{\re}\cap\cC$ is Zariski dense in $\fl_n(f)$.
\end{proof}

\begin{thm}[Hermitian Singul{\"a}rstellensatz]\label{t:herm}
Let $h\in\px^{\ve\times\ve}$ be a full matrix and let $f\in \px^{\de\times\de}$ be an \nice hermitian atom. Then $\fl^{\re}(f)\subseteq\fl^{\re}(h)$ if and only if $f$ is stably associated to a factor of $h$.
\end{thm}

\begin{proof}
Again only $(\Rightarrow)$ is non-trivial. By $\fl^{\re}(f)\subseteq\fl^{\re}(h)$ we have $\fl_n(f)^{\re}\subseteq \fl_n(h)$ for all $n$. Since $f$ is an \nice atom, $\fl_n(f)$ is irreducible and $\fl_n(f)^{\re}$ is Zariski dense in $\fl_n(f)$ for infinitely many $n$ by Proposition \ref{p:nice} and Theorem \ref{t:irr}. Consequently $\fl_n(f)\subseteq \fl_n(h)$ for infinitely many $n$, so $\fl(f)\subseteq\fl(h)$ and Theorem \ref{t:new} applies.
\end{proof}

Similarly to Corollary \ref{c:an}, we can use a modified proof of Proposition \ref{p:inter} to obtain the following.

\begin{cor}\label{c:herm}
For $j=1,\dots,s$ let $f_j\in \px^{\de_j\times\de_j}$ be \nice hermitian atoms and $h\in\px^{\ve\times\ve}$ a full matrix. Then 
$$\fl^{\re}(f_1)\cap\cdots\cap \fl^{\re}(f_s)\subseteq \fl^{\re}(h)$$
if and only if for some $j$, $f_j$ is stably associated to a factor of $h$.
\end{cor}

\begin{exa}\label{e:sos}
Theorem \ref{t:herm} does not hold for arbitrary hermitian atoms $f$.  For example, if $f=x_1x_1^*+x_2x_2^*$ and $h=x_1$, then $\fl^{\re}(f)\subseteq\fl^{\re}(h)$ but $f$ is not stably associated to $h$. For another example, let $f$ be as in \eqref{e:exa}. Then $f(0)=1\oplus -1$ and $f(X,X^*)$ is invertible for every $X$, so $f$ has a constant signature on $\mat{n}$ for every $n$. Moreover, $f$ is not invertible in $\px^{2\times 2}$, and is an atom. Hence $f$ and $h=1$ satisfy $\fl^{\re}(f)=\emptyset=\fl^{\re}(h)$ but $f$ is not stably associated to $h$. For an algorithm checking whether the free real locus of a polynomial is empty, see \cite{KPV}.
\end{exa}

We conclude this section with a Linear Gleichstellensatz for hermitian indecomposable pencils. Since every hermitian monic pencil is unsignatured, the following corollary generalizes \cite[Corollary 5.5]{KV17} and preceding versions with operator-algebraic proofs \cite{HKM13,Zal17,DDOSS17} to non-monic pencils.

\begin{cor}[Hermitian Linear Gleichstellensatz]\label{c:penc}
Let $L,M$ be hermitian indecomposable linear pencils. If $L$ is unsignatured, then $\fl^{\re}(L)=\fl^{\re}(M)$ if and only if $L,M$ are of the same size $d$ and there exists $P\in\GL_d(\C)$ such that $M=\pm PLP^*$.
\end{cor}

\begin{proof}
The implication $(\Leftarrow)$ is obvious, so we consider $(\Rightarrow)$. Since $\fl^{\re}(L)=\fl^{\re}(M)$ and $L,M$ are atoms, $L$ is stably associated to $M$ by Theorem \ref{t:herm}. Therefore $L,M$ are of the same size $d$ and $M=PLQ$ for $P,Q\in\GL_d(\C)$ by Theorem \ref{t:ind}. Since $L,M$ are hermitian, we also have $M=Q^*LP^*$. Therefore $(P^{-1}Q^*,P^*Q^{-1})$ stabilizes $L$, so $P^{-1}Q^*=\alpha I$ by Lemma \ref{l:closed}. Since $M=\alpha PLP^*$ and $L,M$ are hermitian, we have $\alpha\in\R\setminus\{0\}$, so after rescaling $P$ we can choose $\alpha=\pm1$.
\end{proof}


\section{Null- and Positivstellensatz with hard zeros}\label{sec4}

In this section we present a new real Nullstellensatz for hard zeros $\cV^{\re}$ as opposed to determinantal zeros $\fl^{\re}$ discussed above. While the unsignatured condition above is too weak (at least for our techniques), a stronger unsignatured condition succeeds as is seen in Theorem \ref{t:null}. 
Its proof depends on basic commutative algebra and the technique of rational resolvable ideals developed in \cite{KVV}. 
Then, we use this in Corollary \ref{c:posss} to prove a Positivstellensatz (by using a slack variable) for domains defined by quadratic polynomials.

\subsection{Background on rationally resolvable ideals}\label{ss:prelim}

For an ideal $J\subset\ax$ let
$$\cV(J)=\bigcup_{n\in\N}\cV_n(J),\qquad \cV_n(J) = \left\{X\in \mat{n}^g\colon f(X)=0\ \forall f\in J\right\}$$
be its {\bf free zero set}. If $f\in\ax$, then $\cV(f)$ denotes the free zero set of the ideal generated by $f$. We say that $J$ has the {\bf Nullstellensatz property} if
$$f|_{\cV(J)}=0 \quad \iff\quad f\in\ J$$
for all $f\in\ax$. This is a noncommutative analog of a radical ideal in classical algebraic geometry \cite[Section 1.6]{Eis}.

We recall rationally resolvable ideals from \cite{KVV}. The free algebra $\ax$ admits the universal skew field of fractions $\rx$, 
the {\bf free skew field},
whose elements are called {\bf noncommutative rational functions} \cite{Coh,BR,KVV2}. For $1\le h<g$ let $\tilde{x}=(x_1,\dots,x_h)$, and fix a tuple $\rr=(\rr_{h+1},\dots,\rr_g)$ with $\rr_i\in\rtx$. The {\bf graph} of $\rr$ is\looseness=-1
$$\Gamma(\rr)=\bigcup_{n\in\N}\left\{
(\widetilde{X},\rr(\widetilde{X}))\in \mat{n}^g\colon \widetilde{X}\in\bigcap_i\dom\rr_i
\right\}.$$
An ideal $J\subset\ax$ is
\begin{enumerate}
	\item {\bf formally rationally resolvable} with rational resolvent $\rr$ if $J\cap \tx=\{0\}$ and the sets $\{x_{h+1}-\rr_{h+1},\dots,x_g-\rr_g\}$ and $J$ generate the same ideal in the amalgamated product
	$$\ax*_{\C\!\Langle \tilde{x}\Rangle}\,\rtx,$$
	the subring of $\rx$ generated by $\ax$ and $\rtx$.
	\item {\bf geometrically rationally resolvable} with rational resolvent $\rr$ if $\Gamma(\rr)\subseteq\cV(J)$ and for every $h\in \ax$, $h|_{\Gamma(\rr)}=0$ implies $h|_{\cV(J)}=0$.
\end{enumerate}

Intuitively, (1) and (2) allude to relations that can be resolved in a rational manner, either in algebraic or geometric sense. The Nullstellensatz property and rational resolvability are related as follows.

\begin{thm}[{\cite[Theorem 2.5 and Proposition 2.6]{KVV}}]\label{t:older}
	Let $J\subset\ax$ be an ideal. If $\ax/J$ embeds into a skew field and $J$ is formally rationally resolvable with a rational resolvent containing no nested inverses, then $J$ is geometrically rationally resolvable and has the Nullstellensatz property.
\end{thm}

Here, nested inverses refer to presentations of noncommutative rational functions; for example, $(x_1-x_3x_2^{-1}x_4)^{-1}$ cannot be presented without an inverse inside of an inverse, while $(x_1-x_2^{-1})^{-1}=x_2(x_1x_2-1)^{-1}$ admits a presentation without nested inverses. Finally, as in Subsection \ref{ss:real}, we see that if $f=f^*\in\px$ is hermitian, the real structure $\cJ$ preserves $\cV_n(f)\subset\mat{n}^g\times \mat{n}^g$, and its real points are related to the real free zero set of $f$ from the introduction, $\cV_n(f)^{\re}=\left\{(X,X^*)\colon X\in \cV_n^{\re}(f)\right\}$.

\subsection{A real Nullstellensatz for some free zero sets}

The proof of Theorem \ref{t:null} requires a lemma from commutative algebra and an embedding into a free skew field.

\begin{lem}\label{l:irr}
Let $M,L,R\in\opm_n(\C[\omega])$ be such that $\det (MLR) \neq 0$, and consider
$$\cW=\big\{(X,Y)\in \mat{n}^g\times\mat{n}\colon M(X)-L(X)YR(X)=0\big\}.$$
Let $\cW_1\subseteq\cW$ be the union of irreducible components of $\cW$ for which $\det M|_{\cW\setminus\cW_1}=0$ and $\det M$ is not identically zero on any component of $\cW_1$. Similarly, let $\cW_2\subseteq\cW$ be the union of irreducible components of $\cW$ such that $\det (LR)|_{\cW\setminus\cW_2}=0$ and $\det (LR)$ is not identically zero on any component of $\cW_2$. Then $\cW_1=\cW_2$ and $\cW_1$ is irreducible.
\end{lem}

\begin{proof}
By assumption we have $\cW_1\neq\emptyset$. It is clear that $\cW_1\subseteq \cW_2$. Now let $\cZ$ be a component in $\cW\setminus\cW_1$ and suppose $\cZ\not\subseteq \cW\setminus\cW_2$. Then $\cZ$ contains a Zariski dense subset of points $(X,Y)$ with $\det (L(X)R(X))\neq0$. Fix such a point and let $\ve>0$ be arbitrary. Since $Y=L(X)^{-1}M(X)R(X)^{-1}$, there clearly exists $(X_{\ve},Y_{\ve})\in \mat{n}^g\times\mat{n}$ such that
$$\det (M(X_{\ve})L(X_{\ve})R(X_{\ve}))\neq0,\quad
Y_{\ve}=L(X_{\ve})^{-1}M(X_{\ve})R(X_{\ve})^{-1},\quad
\|X-X_{\ve}\|, \|Y-Y_{\ve}\|\le\ve.$$
Therefore $(X_{\ve},Y_{\ve})\in \cW_1$ for every $\ve>0$. Since algebraic sets in a complex affine space are closed with respect to the Euclidean topology, we get $(X,Y)\in\cW_1$. Thus $\cZ\subseteq \cW_1$, a contradiction. Hence $\cW\setminus\cW_1\subseteq\cW\setminus\cW_2$ and therefore $\cW_2\subseteq\cW_1$.

Let $A$ be the coordinate ring of $\cW_1$. Then $\det M$ is not a zero divisor in $A$ by the definition of $\cW_1$, and neither is $\det (LR)$ by the previous paragraph. Therefore $A$ embeds into $A_{\cS}$, where $\cS$ is the multiplicative set generated by $\{\det M, \det L,\det R\}$. Let $J$ be the ideal in $\C[\omega,\upsilon]$ generated by the entries of $M-L\gY R$. Then $\C[\omega,\upsilon] \big/ \sqrt{J}$ is the coordinate ring of $\cW$ 
and by the previous paragraph,
\begin{equation}\label{e:comm}
A_{\cS} =\left(\C[\omega,\upsilon] \big/ \sqrt{J}\right)_\cS
=\C[\omega,\upsilon]_\cS \big/ (\sqrt{J})_\cS
=\C[\omega,\upsilon]_\cS \big/ \sqrt{J_{\cS}}
\end{equation}
since  localization is exact \cite[Proposition 2.5]{Eis} and commutes with the radical \cite[Proposition 2.2, Corollary 2.6 and Corollary 2.12]{Eis}. Note that matrices $L,R$ are invertible over $\C[\omega,\upsilon]_\cS$. Therefore the ideal $J_{\cS}$ in
$$\C[\omega,\upsilon]_\cS = \C[\omega]_{\cS}\otimes_{\C} \C[\upsilon]$$
is generated by the entries of $\gY-L^{-1}MR^{-1}$. By \eqref{e:comm} we thus have $A_\cS\cong\C[\omega]_{\cS}$, so $A$ is an integral domain and $\cW_1$ is irreducible.
\end{proof}

Let $y$ and $y'$ be two freely noncommuting variables.

\begin{lem}\label{l:rr}
Let $f\in\ax$ be nonconstant. Then the ideal $(f-y'y)$ in $\axyy$ is formally rationally resolvable (by pairing $y'$ with the resolvent $fy^{-1}$), and the quotient algebra $\axyy\!/(f-y'y)$ embeds into a free skew field.
\end{lem}

\begin{proof}
First we claim that $(f-y'y)\cap \axy=\{0\}$.  For every element $p\in (f-y'y)$ we have
$$p(X,Y,f(X)Y^{-1})=0$$
for $X\in\mat{n}^g$ and $Y\in\GL_n(\C)$. Hence if $h\in (f-y'y)\cap \axy$, then $h$ is zero on $\mat{n}^g\times \GL_n(\C)$ for every $n\in\N$.
Since $\mat{n}^g\times \GL_n(\C)$ is Zariski dense in $\mat{n}^{g+1}$, we conclude that $h=0$.
Moreover, $f-y'y$ and $fy^{-1}-y'$ clearly generate the same ideal in $\gxy$. Therefore the ideal $(f-y'y)$ is formally rationally resolvable with the rational resolvent $fy^{-1}$. Consider the homomorphism
$$\phi:\axyy\!/(f-y'y)\to \gxy$$
defined by $y'\mapsto f y^{-1}$. Let $W$ be the set of words not containing $y'y$ as a sub-word and let $V=\spa W\subset\axyy$. If $\pi:\axyy\to \axyy\!/(f-y'y)$ is the canonical projection, then it is easy to see that $\pi|_V:V\to \axyy\!/(f-y'y)$ is an isomorphism of vector spaces. On $W$ we define a degree function $\tilde{d}:W\to \Z_{\ge0}^3$ by setting $\tilde{d}(w)=(d_1,d_2,d_3)$, where $d_1$ is the number of $y$'s in $w$, $d_2$ is the number of $y'$'s in $w$, and $d_3$ is the length of $w$. We extend $\tilde{d}$ to a function $V\to\Z_{\ge0}^3$ by
$$\tilde{d}\left(\sum_{w\in W} \alpha_w w\right)=\max_{w\colon \alpha_w\neq0}\tilde{d}(w),$$
where $\Z_{\ge0}^3$ is ordered lexicographically.

As a subalgebra of a free group algebra, $\gxy$ admits a natural basis consisting of reduced words in $x,y,y^{-1}$. For each such word $w$ we define $\hat{d}(w)=(d_1,d_2,d_3)$ analogously as above, where $d_2$ is now the number of $y^{-1}$ in $w$. Similarly as above we obtain the extension $\hat{d}:\gxy\to \Z_{\ge0}^3$.

Since the elements of $\ax\setminus\C$ are freely independent of $y$ and $y^{-1}$, by looking at the highest terms with respect to $\tilde{d}$ and $\hat{d}$ one observes that
$$\hat{d}\big((\phi\circ\pi|_V)(h)\big)=
\left(\tilde{d}(h)_1,\tilde{d}(h)_2,\tilde{d}(h)_3+(\deg f)\tilde{d}(h)_2\right)$$
for every $h\in V$. Therefore $\phi\circ\pi|_V$ is injective, so $\phi$ is an embedding. Hence we are done because the free group algebra embeds into a free skew field (see e.g. \cite[Corollary 7.11.8]{Coh}).
\end{proof}

\begin{prop}\label{p:null}
Let $f\in\ax$ be nonconstant and let $h\in\axyy$. Then there exists $N\in\N$ such that
$$h\left(X,Y,f(X)Y^{-1}\right)=0$$
for all $(X,Y)\in \mat{N}^g\times\GL_N(\C)$ implies $h\in(f-y'y)$. In particular, $(f-y'y)$ has the Nullstellensatz property.
\end{prop}

\begin{proof}
By Lemma \ref{l:rr} and Theorem \ref{t:older}, the ideal $(f-y'y)$ is geometrically rationally resolvable has the Nullstellensatz property. The existence of the bound $N$ then follows by \cite[Theorem 3.9]{KVV}.
\end{proof}

\begin{thm}\label{t:null}
For a nonconstant $f=f^*\in \px$ assume that $f(X_0,X_0^*)\succ0$ for some $X_0$. Let $h\in\pxy$. Then
$\cV^{\re}(f-y^*y)\subseteq \cV^{\re}(h)$ if and only if $h\in(f-y^*y)$.
\end{thm}

\begin{proof}
Let $n\in\N$ be such that $f(X_0,X_0^*)=Y_0^*Y_0$ for some $X_0\in\mat{n}^g$ and $Y_0\in\GL_n(\C)$. Let $\gX$ be a $2g$-tuple of $n\times n$ generic matrices, and $\gY,\gY'$ two additional $n\times n$ generic matrices. Then $\det f(\gX)\neq0$ since $f(X_0,X_0^*)$ is invertible. By Lemma \ref{l:irr} there exists a unique irreducible component $\cW\subset \mat{n}^{2g+2}$ of $\cV_n(f-y^*y)$ such that $\det f(\gX)$ and $\det \gY$ are not identically 0 on $\cW$.

Since $f$ is hermitian, $\cW$ inherits the real structure $\JJ$ from $\mat{n}^{g+1}\times \mat{n}^{g+1}$. Note that the derivative of $f(\gX)-\gY'\gY$ at $(X_0,X_0^*,Y_0,Y_0^*)\in\cW^{\re}$ with respect to $\gy_{\imath\jmath}'$ equals $E_{\imath\jmath}Y_0$, where $E_{\imath\jmath}\in\mat{n}$ are the standard matrix units. The Jacobian matrix of the system of equations $f(\gX)-\gY'\gY=0$ at $(X_0,X_0^*,Y_0,Y_0^*)$ is then equal to
$$\big(
\overbrace{\star \quad \cdots \quad \star}^{(2g+1)n^2}\quad -I_n\otimes Y_0^\ti
\big)\in \C^{n^2\times (gn^2+gn^2+n^2+n^2)}.$$
Therefore $(X_0,X_0^*,Y_0,Y_0^*)\in\cW^{\re}$ is a nonsingular point of $\cW$, so $\cW^{\re}$ is Zariski dense in $\cW$ by Proposition \ref{p:clas}. Because $h$ vanishes on $\cW^{\re}$, it also vanishes on $\cW$. Since $\cW$ is the unique component of $\cV_n(f-y'y)$ on which $\det \gY$ does not constantly vanish, we have 
$$h\left(X,Y,f(X)Y^{-1}\right)=0$$
for all $(X,Y)\in \mat{n}^{2g}\times\GL_n(\C)$. Since $n$ can be taken arbitrarily large, we have $h\in (f-y^*y)$ by Proposition \ref{p:null}.
\end{proof} 

\subsection{A Positivstellensatz for hereditary quadratic polynomials}

Let $f\in\px$ be a hermitian hereditary quadratic polynomial. That is,
$$f=\alpha +\vec{x}^{\,*}v+v^*\vec{x}+ \vec{x}^{\,*} H\vec{x},$$
where $H$ is a hermitian $g\times g$ matrix, 
$\vec x$ is the column vector consisting of the variables $x_j$,
$v\in\C^g$ and $\alpha\in\R$. Then $\{f\succ0\}\neq\emptyset$ if and only if $-f$ is not a sum of (hermitian) squares (this is not true for more general polynomials, e.g., $x_1x_1^*-x_1^*x_1$).

\begin{cor}\label{c:posss}
Let $f\in\px$ be a nonconstant hermitian hereditary quadratic polynomial with $\{f\succ0\}\neq\emptyset$, and let $y$ be an auxiliary noncommuting variable. If $h\in\px$, then $h|_{\{f\succeq0\}}\succeq0$ 
if and only if
\begin{equation}\label{e:c2}
h=f_0+\sum_j f_j^*f_j
\end{equation}
for some $f_j\in \pxy$ with $f_0\in (f-y^*y)$.
\end{cor}

\begin{proof}
$(\Rightarrow)$ If $h|_{\{f\succeq0\}}\succeq0$, then $h|_{\cV(f-y^*y)}\succeq 0$, where $h$ is viewed as an element of $\pxy$. Since $f-y^*y$ is also a quadratic hereditary polynomial, there exist $f_j\in \pxy$ with $f_0|_{\cV(f-y^*y)}=0$ such that
$$h=f_0+\sum_j f_j^*f_j$$
by \cite[Theorem 4.1 and Section 4.2.c]{HMP} (or rather its version over $\C$). Moreover, $f_0\in (h-y^*y)$ by Theorem \ref{t:null}.

$(\Leftarrow)$ Suppose \eqref{e:c2} holds. Let $X$ be such that $f(X,X^*)\succeq0$. If $Y^*Y=f(X,X^*)$ then $f_0(X,X^*,Y,Y^*)=0$ and hence $h(X,X^*)\succeq0$.
\end{proof}

\begin{rem}
In general, $h|_{\{f\succeq0\}}\succeq0$ does not imply that $h$ is of the form
$$\sum_j f_j^*f_j+\sum_kg_k^* fg_k,$$
i.e., $h$ does not necessarily belong to the quadratic module generated by $f$. An example with $f=1-x^*x$ and $h=1-\frac43 x^*x+\alpha x^*xx^*x$ for $\frac13<\alpha<\frac49$ is a consequence of \cite[Section 3.1]{DAP}. Thus the slack variable $y$ of Corollary \ref{c:posss} is necessary.
\end{rem}

\begin{rem}
Corollary \ref{c:posss} is a rare example of an ``if and only if'' noncommutative Positivstellensatz.
Another one appears in \cite{HKM12} and applies to quadratic  $f$ whose positivity set is convex. For such $f$ our Corollary \ref{c:posss}
can be readily proved as a consequence of \cite[Theorem 1.1(1)]{HKM12}.
\end{rem}


\end{document}